\documentclass[twocolumn,10pt]{ieeeconf}
\usepackage{graphicx}
\usepackage{algorithm,algorithmic}	
\usepackage{amssymb}
\usepackage{amsmath,todonotes}
\usepackage{epstopdf}

\usepackage{hyperref}

\newtheorem{definition}{Definition}
\newtheorem{theorem}{Theorem}

\newtheorem{lemma}{Lemma}

\title{Topology Design for Optimal Network Coherence}
\author{Tyler Summers, Iman Shames, John Lygeros, Florian D\"{o}rfler
\thanks{T. Summers, F. D\"{o}rfler, and J. Lygeros are with the Automatic Control Laboratory, ETH Z\"{u}rich. I. Shames is with the Department of Electrical Engineering, University of Melbourne.}}
\date{\today}                                           

\begin{document}
\maketitle

\begin{abstract}
We consider a network topology design problem in which an initial undirected graph underlying the network is given and the objective is to select a set of edges to add to the graph to optimize the coherence of the resulting network. We show that network coherence is a submodular function of the network topology. As a consequence, a simple greedy algorithm is guaranteed to produce near optimal edge set selections. We also show that fast rank one updates of the Laplacian pseudoinverse using generalizations of the Sherman-Morrison formula and an accelerated variant of the greedy algorithm can speed up the algorithm by several orders of magnitude in practice. These allow our algorithms to scale to network sizes far beyond those that can be handled by convex relaxation heuristics.
\end{abstract}

\section{Introduction}
Among the most challenging and important problems in control and optimization of networks of dynamical systems is the design of topologies for sensing, control, and communication. One important dynamical process in a variety of networks is synchronization, and it is widely recognized that network topology properties play a key role.
There is now a growing literature on synchronization and consensus and many associated applications, including power networks and robotic vehicle networks. 

Recently, there has been work on quantifying the robustness of consensus dynamics to stochastic disturbances \cite{bamieh2012coherence,BB-DG:13,SM-NM:14}. The concept of network coherence has been proposed to quantify variance of states around the consensus subspace \cite{bamieh2012coherence} and is also closely related to the effective resistance of graphs \cite{klein1993resistance,ghosh2008minimizing}. Several recent papers have focused on leader selection problems to optimize coherence\cite{patterson2010leader,patterson2011network,lin2011algorithms,fardad2011algorithms,lin2013algorithms,fitch2013information,bushnell2014supermodular}. However, designing network topologies to optimize coherence by choosing sets of edges has received less attention. In \cite{xiao2007distributed,ghosh2008minimizing} the problems of choosing edge weights for a given network topology to optimize network coherence and effective resistance are considered and shown to be convex optimization problems. It is possible to modify the algorithm in \cite{xiao2007distributed,ghosh2008minimizing} to obtain an convex relaxation heuristic for edge selection. However, the resulting problems can still be difficult to solve for large networks, and purely combinatorial versions of these problems appear not to have been considered.

We consider a network topology design problem in which an initial undirected graph underlying the network is given and the objective is to select a set of edges to add to the graph to optimize the coherence of the resulting network. Our main result is to show that network coherence is a submodular set function of the network topology. As a consequence, a simple greedy algorithm is guaranteed to produce near optimal edge set selections. The problem has a similar mathematical structure to recently studied set function optimization problems linking submodularity to controllability \cite{summers2014submodularity,summers2014optimal,cortesi2014submodularity} and rigidity \cite{shames2014rigid}. We also show that fast rank one updates of the Laplacian pseudoinverse using generalizations of the Sherman-Morrison formula \cite{meyer1973generalized} and an accelerated variant of the greedy algorithm \cite{minoux1978accelerated,krause2012submodular} can speed up the algorithm by several orders of magnitude in practice. These techniques allow our algorithms to scale to network sizes far beyond those that can be handled by convex relaxation heuristics based on \cite{xiao2007distributed,ghosh2008minimizing}.
The results are illustrated with numerical examples. 

The rest of the paper unfurls as follows. Section II provides background on network coherence and submodular set functions. Section III presents our main results on topology design and algorithmic speed ups. Section IV exhibits the performance of the accelerated algorithm and illustrates the results on various types of fixed and random initial networks. Finally, Section V concludes.

\section{Network coherence and submodular set functions}
This section reviews notions of network coherence and submodular set functions. Topology design problems for optimizing network coherence can be formulated as set function optimization problems. 

\subsection{Network coherence}
Consider a network with underlying weighted undirected graph $G = (V,E,w)$ where $V = \{1,...,n \}$ is a set of nodes, $E \subseteq V \times V$ is a set of edges, and $w \in \mathbf{R}^{|E|}$ is a set of nonnegative weights associated with each edge. Suppose a scalar state variable is associated with each node and the network has consensus dynamics modeled by the stochastic differential equation 
\begin{equation} \label{sde}
dx(t) = -Lx(t)dt + dW
\end{equation}
where $L$ is the weighted Laplacian matrix and $dW$ is a vector of independent Gaussian white noise stochastic processes.

Without noise, it is well known that when the graph is connected, the states converge exponentially to a point the consensus subspace corresponding to the average value of the initial conditions. With additive noise, the state evolution becomes a stochastic process. The expected values of the states evolve according to deterministic consensus dynamics, but the actual average value undergoes Brownian motion and the states stochastically fluctuate around this value. When the graph is connected, the state variance relative to the average converges to a steady state value. Network coherence quantifies the steady-state variance of these fluctuations and can be considered as a measure of robustness of the consensus process to the additive noise; networks with small steady-state variance have high network coherence and can be considered to be more robust to noise that networks with low coherence.

Formally, network coherence is defined for connected graphs as the average steady-state deviation from the average value
\begin{equation}
\begin{aligned}
\mathcal{C} &= \lim_{t \rightarrow \infty} \sum_{i = 1}^n \mathbf{E} \left[(x_i(t) - \frac{1}{n}\sum_{j=1}^n x_j(t))^2\right] \\
&= \lim_{t \rightarrow \infty} \mathbf{E} [ x(t)^T P x(t)],
\end{aligned}
\end{equation}
where $P = I - \frac{1}{n} \mathbf{1} \mathbf{1}^T$ is the projection operator onto the disagreement subspace. The coherence then relates to the system $H_2$ norm as
\begin{equation}
\mathcal{C} = \textbf{trace} \int_0^\infty e^{-L^T t} P e^{-Lt} dt.
\end{equation}
This in turn can be shown to be related to the spectrum of the Laplacian matrix \cite{patterson2010leader}
\begin{equation}
\mathcal{C} = \frac{1}{2}\textbf{trace}(L^\dag) = \frac{1}{2}\sum_{i=2}^n \frac{1}{\lambda_i(L)},
\end{equation}
where $0 = \lambda_1 < \lambda_2 \leq ... \leq  \lambda_n$, i.e., the pseudoinverse trace of the Laplacian matrix is proportional to network coherence. This quantity is also proportional to the total effective resistance of a graph \cite{klein1993resistance}, which is known to be a monotone convex function of the edge weights for a given topology \cite{ghosh2008minimizing}.

\subsection{Submodularity}
Many combinatorial problems can be formulated as \emph{set function} optimization problems. For a given finite set $V = \{1,...,M \}$, a \emph{set function} $f: 2^V \rightarrow \mathbf{R}$ assigns a real number to each subset of $V$. For the set function optimization problem
\begin{equation} \label{optprob}
 \underset{{S \subseteq V, \ |S| = k }}{\text{maximize}} \quad f(S),
\end{equation}
the objective is to select a $k$-element subset of $V$ that maximizes $f$. This can be solved by brute force by  enumerating all possible subsets of size $k$, evaluating $f$ for all of these subsets, and picking the best subset. However, the number of possible subsets grows factorially as $|V|$ increases, so the brute force approach quickly becomes infeasible even for moderate $|V|$. 

Instead, there are structural properties of the set function $f$ that facilitate optimization. In particular, \emph{submodularity} plays similar roles in combinatorial optimization as convexity and concavity play in continuous optimization \cite{lovasz1983submodular,krause2012submodular}. It occurs often in applications \cite{boykov2001interactive,kempe2003maximizing,krause2008near}; is supported by an elegant and practically useful mathematical theory; and there are efficient methods for minimizing and approximation guarantees for maximizing submodular functions. 

\begin{definition}[Submodularity]
A set function $f: 2^V \rightarrow \mathbf{R}$ is called submodular if for all subsets $A \subseteq B \subseteq V$ and all elements $s \notin B$, it holds that
\begin{equation} \label{submod1}
f(A \cup \{s\}) - f(A) \geq f(B \cup \{s\}) - f(B),
\end{equation}
or equivalently, if for all subsets $A,B \subseteq V$, it holds that
\begin{equation} \label{submod2}
f(A) + f(B) \geq f(A\cup B) + f(A\cap B).
\end{equation}
\end{definition}
Intuitively, submodularity is a diminishing returns property where adding an element to a smaller set gives a larger gain than adding one to a larger set. 
The following definition and result from \cite{lovasz1983submodular} makes this precise and will be used to prove submodularity of a set function associated with network coherence.
\begin{definition}
A set function $f: 2^V \rightarrow \mathbf{R}$ is called monotone increasing if for all subsets $A, B \subseteq V$ it holds that
\begin{equation}
A \subseteq B \Rightarrow f(A) \leq f(B)
\end{equation}
and is called monotone decreasing if for all subsets $A, B \subseteq V$ it holds that
\begin{equation}
A \subseteq B \Rightarrow f(A) \geq f(B).
\end{equation}
\end{definition}
\begin{theorem}[\cite{lovasz1983submodular}] \label{theorem:mono}
A set function $f: 2^V \rightarrow \mathbf{R}$ is submodular if and only if the derived set functions $f_a : 2^{V - \{ a \} } \rightarrow \mathbf{R}$
$$f_a (X) = f(X \cup \{a\}) - f(X)  $$ 
are monotone decreasing for all $a \in V$.
\end{theorem}


Maximization of monotone increasing submodular functions is NP-hard, but a greedy heuristic can be used to obtain a solution that is provably close to the optimal solution \cite{greedybound}. The greedy algorithm for \eqref{optprob} starts with an empty set, $S_0 \leftarrow \emptyset$, computes the gain $\Delta(a \mid {S_i}) = f(S_{i}\cup \{a\})-f(S_{i})$ for all elements $a\in V\backslash S_{i}$ and adds any element with the highest gain: \[S_{i+1} \leftarrow S_{i}\cup \{\arg \max_{a} \Delta(a \mid {S_i}) \; |\; a\in V\backslash S_{i}\}. \]
The algorithm terminates after $k$ iterations.

Performance of the greedy algorithm is guaranteed by a well known bound~\cite{greedybound}:
\begin{theorem}[\cite{greedybound}]
Let $f^*$ be the optimal value of the set function optimization problem \eqref{optprob}, and let $f(S_{greedy})$ be the value associated with the subset $S_{greedy}$ obtained from applying the greedy algorithm on \eqref{optprob}. If $f$ is submodular and monotone increasing, then 
\begin{equation}
	\frac{f^* - f(S_{greedy})}{f^* - f(\emptyset)}
	\leq \left(\frac{k-1}{k}\right)^k
	\leq \frac{1}{e}
	\approx 0.37.
\label{eq:greedy_bound}\end{equation}
\end{theorem}
This means that the greedy algorithm is guaranteed to produce a subset whose function value is within a constant factor of the value of the optimal subset. This is the best any polynomial time algorithm can achieve \cite{feige1998threshold}, assuming $P\neq NP$. Note that this is a worst-case bound; the greedy algorithm often performs much better than the bound in practice, which we will verify for the considered set of problems.

\section{Optimal topology design for network coherence}
Consider the problem of choosing a subset $\mathcal{E}$ of $k$ edges, each with a given weight, to add to a given weighted undirected graph $G = (V,E)$ to maximize the network coherence of the resulting graph $G_\mathcal{E}$, which can be formulated as a set function optimization problem:
\begin{equation}\label{eq:cohopt}
\underset{\mathcal{E} \subset V \times V \setminus E}{\text{minimize}} \quad  \mathbf{trace}(L_\mathcal{E}^\dag),
\end{equation}
where $L_\mathcal{E}$ is the resulting Laplacian. In the first subsection, we will assume that the given graph is connected so that the rank of the Laplacian remains constant as edges are added. In the third subsection, we will relax this assumption and present a modified algorithm for constructing trees with good network coherence.

\subsection{Network coherence is a submodular function of network topology}
To prove that network coherence is a submodular function of the network topology, the key structure is additivity of the Laplacian matrix in the edges. Specifically, let $M\in \mathbf{R}^{|E| \times |V|} $ denote the weighted incidence matrix of a graph $G = (V,E,w)$ which has a row $m_e^T$ for each edge $e = (i,j) \in E$ with elements (for $i>j$) $M_{ev} = w_e$ if $v = i$, $M_{ev} = -w_e$ if $v = j$, and $M_{ev} = 0$ otherwise. The Laplacian associated with any edge set $E$ can be written 
\begin{equation}
L_E = M^T M =  \sum_{e=1}^{|E|} m_e m_e^T = \sum_{e=1}^{|E|} L_e,
\end{equation}
which implies $L_{E_1 \cup E_2} = L_{E_1} + L_{E_2}$ for any disjoint pair of edge sets $E_1$ and $E_2$. We have the following result; the proof has almost identical structure to the proof for the inverse of the controllability Gramian in \cite{summers2014submodularity} and the proof for the pseudoinverse of the rigidity Gramian in \cite{shames2014rigid}.

\begin{theorem} \label{netcohsubmod}
Let $G=(V,E,w_E)$ be a given connected weighted graph, let $\mathcal{E} \subseteq V\times V \setminus E$ with weights $w_\mathcal{E}$, and let $L_{\mathcal{E}}$ be the weighted graph Laplacian matrix associated with the edge set $E  \cup \mathcal{E}$. Then the set function $f: V\times V \setminus E \rightarrow \mathbf{R}$ defined by $f(\mathcal{E}) = -\mathbf{trace}(L_{\mathcal{E}}^\dag)$ is submodular.
\end{theorem}

\begin{proof} Denote the set of potential edge choices by $\mathcal{E}_c = V \times V \setminus E$. Take any $e \in \mathcal{E}_c$ and consider the derived set functions $f_e : 2^{\mathcal{E}_c\setminus \{e\} } \mapsto \mathbb{R}$ given by
\begin{equation*}
\begin{aligned}
f_e(\mathcal{E}) &=  -\textbf{trace}(L_{\mathcal{E}\cup \{e\} }^\dag) + \textbf{trace}(L_{\mathcal{E}}^\dag) \\
            &= -\textbf{trace}((L_{\mathcal{E}}+L_{e})^\dag) + \textbf{trace}(L_{\mathcal{E}}^\dag).
\end{aligned}
\end{equation*}
Take any $\mathcal{E}_1 \subseteq \mathcal{E}_2 \subseteq \mathcal{E}_c \setminus \{ e\}$. By the additivity property of the Laplacian, it is clear that $\mathcal{E}_1 \subseteq \mathcal{E}_2 \Rightarrow L_{\mathcal{E}_1} \preceq L_{\mathcal{E}_2}$. Now define $L(t) = L_{\mathcal{E}_1} + t(L_{\mathcal{E}_2} - L_{\mathcal{E}_1})$ for $t \in [0,1]$. Obviously, $L(0) = L_{\mathcal{E}_1}$ and $L(1) = L_{\mathcal{E}_2}$. Now define
\begin{equation*}
\hat{f}_e(L(t)) =  -\textbf{trace}((L(t) + L_e )^\dag) + \textbf{trace}(L(t) ^\dag).
\end{equation*}
Note that $\hat{f}_e(L(0)) = f_e(\mathcal{E}_1)$ and $\hat{f}_e(L(1)) = f_e(\mathcal{E}_2)$. We have
\begin{equation*}
\begin{aligned}
&\frac{d}{dt}\hat{f}_e\left(L(t)\right) =\frac{d }{d t}\left[- \textbf{trace}((L(t)+L_{e})^\dag) + \textbf{trace}(L(t)^\dag) \right]\\
&=\textbf{trace} \left[(L(t)+L_{e})^\dag(L_{\mathcal{E}_2} - L_{\mathcal{E}_1})(L(t)+L_{e})^\dag \right] \\ &\quad- \textbf{trace} \left[L(t)^\dag(L_{\mathcal{E}_2} - L_{\mathcal{E}_1})L(t)^\dag\right]\\
&=\textbf{trace} \bigg[ \left((L(t)+L_{e})^{\dag , 2}-L(t)^{\dag , 2}\right)(L_{\mathcal{E}_2} - L_{\mathcal{E}_1}) \bigg] \leq 0.
\end{aligned}
\end{equation*}

To obtain the second equality we used the matrix derivative formula $\frac{d}{d t}  \textbf{trace} (L(t)^\dag)  =  \textbf{trace} (L(t)^\dag \frac{d}{d t}( L(t)) L(t)^\dag )$ which holds whenever $L(t)$ has constant rank for all $t$ \cite{golub1973differentiation}, which we have here since the given graph is connected and thus $\text{rank}(L_\mathcal{E}) = n-1$, $\forall \mathcal{E} \subseteq \mathcal{E}_c$. To obtain the third equality we used the cyclic property of trace. Since $(L(t)+L_{e})^{\dag ,2}-L(t)^{\dag ,2} \preceq 0$ and $L_{\mathcal{E}_2} - L_{\mathcal{E}_1} \succeq 0$, the last inequality holds because  the trace of the product of a positive and negative semidefinite matrix is non-positive.
Since
\begin{equation*}
\hat{f}_e(L(1)) = \hat{f}_e(L(0)) + \int_0^1 \frac{d}{d t} \hat{f}_e(L(t)) dt,
\end{equation*}
it follows that $\hat{f}_e(L(1)) = f_e(\mathcal{E}_2) \leq \hat{f}_e(L(0)) = f_e(\mathcal{E}_1)$. Thus, $f_e$ is monotone decreasing, and $f$ is submodular by Theorem \ref{theorem:mono}.

Finally, it can be seen from additivity of the Laplacian that $f$ is monotone increasing, which just means that adding an edge to the graph cannot decrease its coherence. 
\end{proof}

As a consequence, the greedy algorithm is guaranteed to produce a near optimal edge set selection. If the given graph is not connected, the Laplacian changes rank as edges are added. This means that $L(t)$ in the proof does not have constant rank, so $\hat{f}_e$ is not differentiable, and the proof breaks down.

\subsection{Accelerated greedy algorithm and fast rank-one updates} \label{fastgreedy}
For a sparse connected network, the number of possible edges to be added scales quadratically with the number of nodes. So for the standard greedy algorithm, the marginal gain function may need to be evaluated many times. 

Two techniques can be used to significantly speed up the greedy algorithm. First, an accelerated form of the greedy algorithm can be used to reduce the number of times that the marginal gain function is evaluated by exploiting submodularity of the set function \cite{minoux1978accelerated}. In particular, at each iteration an element is selected that maximizes the marginal benefit of the function given previously chosen elements. The key observation is that submodularity implies $\Delta(s | S_{i+1} ) \leq \Delta(s | S_i)$; i.e., the marginal benefits of each element can never increase between algorithm iterations. In the accelerated greedy algorithm, after the first iteration, a list of marginal benefits from the previous iteration sorted in decreasing order is maintained. The marginal benefits for the next iteration are then updated starting from the top of this list. If during this process an element remains at the top of the list after this update, submodularity guarantees that this element has maximal marginal gain, and the algorithm can move to the next iteration without needing to compute the marginal gain for a potentially very large number of elements. Otherwise, the list is resorted and the process continues. Although the worst case complexity of this accelerated variant is the same as the naive greedy algorithm, speedups of multiple orders of magnitude have been observed in practice \cite{krause2012submodular}.

Second, the individual marginal gain function calls can be cheaply performed as rank-one updates using a generalized Sherman-Morrison formula. Note that computing the marginal gain requires computing the trace of the pseudoinverse of a matrix following a rank one update. Although the standard Sherman-Morrison formula does not hold in general for updating the pseudoinverse of a matrix (as opposed to the inverse) \cite{meyer1973generalized}, we show that for the case of the Laplacian pseudoinverse, one can use the standard formula with the  inverse replaced by the pseudoinverse. 
\begin{lemma} \label{smupdate}
For any connected weighted graph $G = (V,E,w)$ with weighted Laplacian matrix $L_E$ and any edge $e \in V\times V \setminus E$ with given weight $w_e$ and associated weighted incidence matrix row $m_e$, we have
\begin{equation} \label{smpseudo}
L_{E \cup \{e\}}^\dag = (L_E + m_e m_e^T)^\dag = L_E^\dag - \frac{1}{\beta} L_E^\dag m_e m_e^T L_E^\dag, 
\end{equation}
where $\beta = 1 + m_e^T L_E^\dag m_e$, and correspondingly,
\begin{equation} \label{smpseudotrace}
\mathbf{trace}(L_{E \cup \{e\}}^\dag) = \mathbf{trace}(L_E^\dag) - \frac{1}{\beta} || L_E m_e ||^2.
\end{equation}
\end{lemma}
\begin{proof}
From Theorem 3 in \cite{meyer1973generalized}, it holds that
$$ (L_E + m_e m_e^T)^\dag = L_E^\dag + \frac{1}{\beta}u k^T L_E^{\dag} - \frac{\beta}{\sigma}pq,  $$
where $ u = (I - L_E L_E^\dag)m_e$, $k = L_E^\dag m_e$, $\beta = 1 + m_e^T L_E^\dag m_e$, $p = \frac{||k||^2}{\beta}u + k$, $q = (\frac{||u||^2}{\beta}k^T L_E^\dag + k^T)$, and $\sigma = || k ||^2 || u ||^2 + \beta^2$. But since $I - L_E L_E^\dag = \frac{1}{n}\mathbf{1}\mathbf{1}^T$, it follows that for any incidence matrix row, $(I-LL^\dag)m_e = u = 0$, so the expression collapses immediately to \eqref{smpseudo}. Finally, \eqref{smpseudotrace} follows from the linearity and cyclic properties of trace. 
\end{proof}
This means that the full Laplacian pseudoinverse needs to be computed from scratch only once at the beginning for the given connected graph, requiring $\mathcal{O}(n^3)$ operations. Then within each iteration, \eqref{smpseudotrace} can be used to evaluate the marginal gain for each edge in $\mathcal{O}(n)$ operations since $m_e$ has only two non-zero entries. Finally, after the optimizing edge has been found, the Laplacian pseudoinverse can by updated using \eqref{smpseudo} in $\mathcal{O}(n^2)$ operations. 

\subsection{Constructing Tree Graphs with Optimal Coherence and Nonidentical Edge Weights} \label{tree}
We now relax the assumption that the given graph is connected and consider the problem of constructing a tree graph with optimal network coherence. It is known that when the edge weights are identical, a star graph is the tree with optimal network coherence; see, e.g., \cite{ellens2011effective}. However, when non-identical edge weights are given, it is not obvious how to find the tree with optimal coherence. Given a node set $V=\{1,\dots,n\}$ and weights $w_e \geq 0$ associated with each possible edge $e \in V \times V$, the goal is to find an edge set $\mathcal{E}$ such that the undirected graph $\mathcal{G}=(V,\mathcal{E})$ is a connected tree with minimum $\mathbf{trace}(L_\mathcal{E}^\dag)$. 
This can be expressed as
\begin{equation}\label{eq:tree_const}
\begin{split}
\underset{\mathcal{E} \subset V\times V }{\text{minimize}} \quad &  \mathbf{trace}(L_\mathcal{E}^\dag)\\
\text{subject to} \quad & | \mathcal{E}|=n-1, \quad \mathbf{rank}(L_\mathcal{E})=n-1.
\end{split}
\end{equation}

A modified greedy algorithm detailed in Algorithm \ref{alg1-tree-lapinv} can be used as a heuristic for this problem. 
\begin{algorithm}                      
\caption{Finding a tree with small $ \mathbf{trace}(L_\mathcal{E}^\dag)$.}          
\label{alg1-tree-lapinv}                           
\begin{algorithmic}[1]                
    \REQUIRE $V, \; w_e$ for each $e \in V\times V$
    \STATE $\mathcal{E} \leftarrow \{ \mathop{\text{argmax}}_{e\in V\times V} w_e \}$
    \STATE $\overline{\mathcal{V}} \leftarrow \{i,j|(i,j)\in \mathcal{E}\}$
    \WHILE{$|\mathcal{E}| \leq n-1$}
    	\STATE $\overline{\mathcal{E}} \leftarrow \{(i,j) \mid (i,j)\in V\times V \setminus \mathcal{E}, \{i,j\}\cap \overline{\mathcal{V}}\neq \emptyset\}$
        \STATE $e = \mathop{\text{argmin}}_{e\in \overline{\mathcal{E}}} \quad \mathbf{trace}(L_{\mathcal{E}\cup \{e\}}^\dag)$
        \STATE $\mathcal{E}\leftarrow \mathcal{E} \cup e$
        \STATE $\overline{\mathcal{V}} \leftarrow \overline{\mathcal{V}} \cup \{i,j \mid e=(i,j)\}$
    \ENDWHILE
\end{algorithmic}
\end{algorithm}
The only difference is that the feasible edge set described in line 4 prevents cycles from forming, which would involve redundant edges in terms of connectivity. As in Lemma \ref{smupdate}, the value of $\mathbf{trace}(L_{\mathcal{E}\cup \{e\}}^\dag)$ can be calculated more efficiently. From Theorem 1 of \cite{meyer1973generalized}
\begin{equation}
\begin{split}
&\mathbf{trace}(L_{\mathcal{E}\cup \{e\}}^\dag)=\mathbf{trace}(L_\mathcal{E} + m_e m_e^T)^\dag \\
&=  \mathbf{trace} (L_\mathcal{E}^\dag) -\dfrac{\mathbf{trace} \Big (L_\mathcal{E}^\dag  m_e m_e^T (I-L_\mathcal{E} L_\mathcal{E}^\dag)\Big )}{\| (I-L_\mathcal{E} L_\mathcal{E}^\dag) m_e\|^2}\\
&-\dfrac{ \mathbf{trace}\Big ( (I-L_\mathcal{E} L_\mathcal{E}^\dag)  m_e m_e^T L_\mathcal{E}^\dag \Big )}{\| (I-L_\mathcal{E} L_\mathcal{E}^\dag)m_e\|^2}
+\dfrac{1+ m_e^T L_\mathcal{E}^\dag m_e}{\| (I-L_\mathcal{E} L_\mathcal{E}^\dag)m_e\|^2}.
\end{split}
\end{equation}
Since $L_\mathcal{E}^\dag L_\mathcal{E} L_\mathcal{E}^\dag=L_\mathcal{E}^\dag$, we have
$
\mathbf{trace}\Big ( (I-L_\mathcal{E} L_\mathcal{E}^\dag) m_e m_e^T L_\mathcal{E}^\dag \Big )=\mathbf{trace}\Big (  L_\mathcal{E}^\dag(I-L_\mathcal{E} L_\mathcal{E}^\dag) m_e m_e^T \Big )=0.
$
Hence,
\begin{equation}
\begin{split}
\mathbf{trace}(L_{\mathcal{E}\cup \{e\} }^\dag)&=  \mathbf{trace}(L_\mathcal{E}^\dag)+\dfrac{1+ m_e^T L_\mathcal{E}^\dag m_e}{\| (I-L_\mathcal{E} L_\mathcal{E}^\dag)m_e\|^2}.
\end{split}
\end{equation}
Thus, the marginal gain computations in line 5 of Algorithm \ref{alg1-tree-lapinv} can be written as
$$e = \mathop{\text{argmin}}_{e\in \overline{\mathcal{E}}} \quad \dfrac{1+ m_e^T L_\mathcal{E}^\dag m_e}{\| (I-L_\mathcal{E} L_\mathcal{E}^\dag)m_e\|^2}.$$


Note that Algorithm \ref{alg1-tree-lapinv} can be used to add a new node or a set of new vertices to an existing graph such that resulting graph with a  small $ \mathbf{trace}(L_\mathcal{E}^\dag)$ as well.

In the Appendix, we show that Algorithm \ref{alg1-tree-lapinv} returns a star graph when the edge weights are identical, which provides an alternative, and to the authors' knowledge novel, inductive proof of the fact that star graphs have optimal network coherence among all trees. 

\section{Illustrative numerical examples}
In this section, we illustrate the results with numerical examples.

\subsection{Naive vs. fast greedy algorithm}
We first compare the performance of the naive greedy algorithm with that of the modified greedy algorithm using the improvements described in Section \ref{fastgreedy}. We applied both algorithms to compute a set of $n$ edges to add to $n$-node Erd\H{o}s-R\'{e}nyi random networks with the edge probability chosen to be slightly above $\ln (n) /n$ to ensure connectivity of the generated base graph. Note that even for the smallest networks considered here, e.g. choosing 20 edges to add from around 150, brute force computation is not feasible. However, the greedy algorithm is guaranteed by Theorem \ref{netcohsubmod} to produce a network topology with sub optimality-guarantees. Figure \ref{fig:fastvnaive} shows computation times for the two algorithms for various network sizes on a laptop with a 1.7 GHz Intel Core i7 processor. One can see a substantial increase in computation time for the naive greedy algorithm around 100 nodes (corresponding to about 4500 possible edges to add). This is also roughly where convex relaxation heuristics based on \cite{xiao2007distributed,ghosh2008minimizing} and using general purpose semidefinite programming solvers begin to have difficulties. The fast greedy algorithm displays significantly better scaling properties: for the data at 120 nodes, the fast algorithm exhibits a factor of 350 speed-up. Our techniques (using unoptimized Python code) were able to near-optimally modify networks with up to 1,000 nodes, with nearly half a million decision variables associated possible edges, in a few hours, which is far beyond the capabilities of current state-of-the-art general purpose semidefinite programming solvers.  
\begin{figure}
\begin{center} 
\resizebox{\linewidth}{!}{\includegraphics{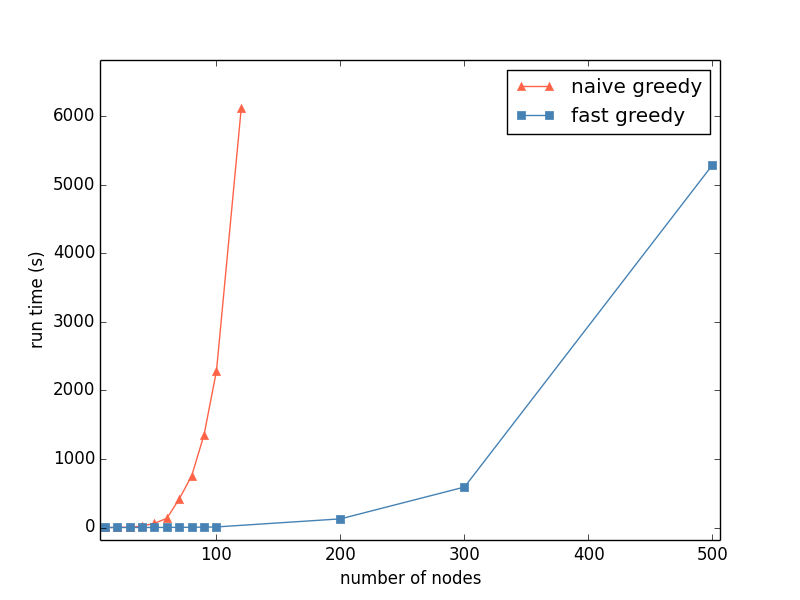}}
\caption{Computation times for the naive and fast versions of the greedy algorithm. The algorithms were applied to compute a set of $n$ edges to add to $n$-node Erd\H{o}s-R\'{e}nyi random networks with the edge probability chosen to be slightly above $\ln n /n$ to ensure connectivity of the generated base graph. For the data at 120 nodes, the fast algorithm exhibits a factor of 350 speed-up.}  \label{fig:fastvnaive}
\end{center}
\end{figure}

\subsection{Experiments with cycles and random graphs}
We then examined the qualitative behavior of the algorithm for cycles and other types of random graphs. Figure \ref{fig:cycle} shows the results of applying to greedy algorithm to add various numbers of edges to a cycle on 50 nodes. The initial added edges tend to be long distance links, reminiscent of Watts-Strogatz small world graphs \cite{watts1998collective}, but with the link distances intentionally chosen by the algorithm to optimize coherence. When many edges are added, the resulting graph tends to be nearly regular, indicating that small world regular graphs have near optimal coherence. A similar story emerges for Erd\H{o}s-R\'{e}nyi random graphs. Figure \ref{fig:erdos} shows the result of applying the greedy algorithm. In small added edge sets, the added edges tend to connect distant low-degree vertices, and in large added edge sets, the result tends to a regular graph with small-world-like long distance connections. Finally, we also applied the greedy algorithm to  Barabasi-Albert scale free networks, in which a preferential attachment mechanism leads to power law degree distributions. Figure \ref{fig:ba} shows a set of 10 edges added to a scale free tree on 100 nodes. We observe that some of the added edges tend to connect highly connected hubs together, while others make low-degree long-distance connections.

\begin{figure}
\begin{center} 
\resizebox{\linewidth}{!}{\includegraphics{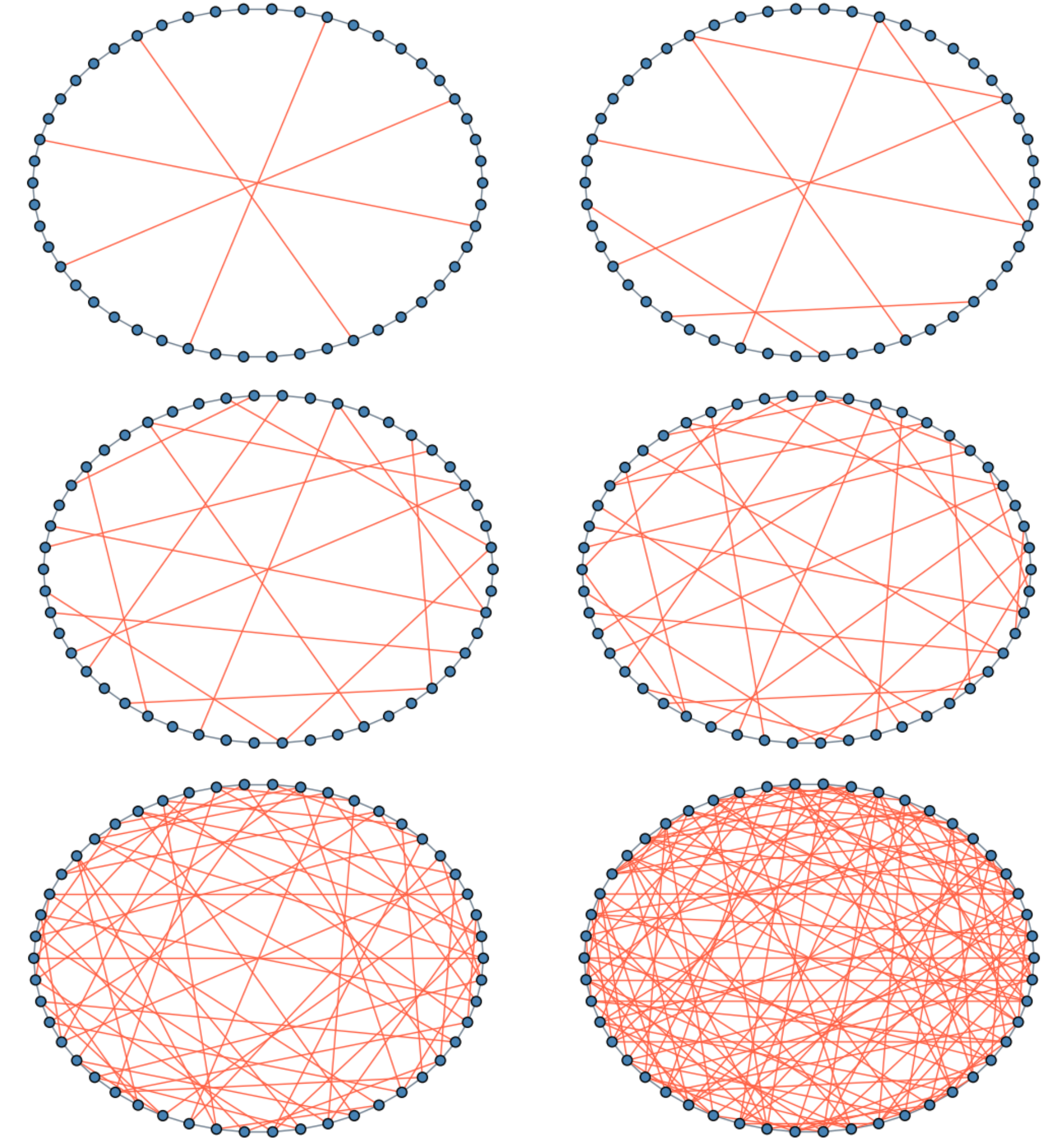}}
\caption{Adding various numbers of edges to a cycle graph on 50 nodes. }  \label{fig:cycle}
\end{center}
\end{figure}

\begin{figure}
\begin{center} 
\resizebox{\linewidth}{!}{\includegraphics{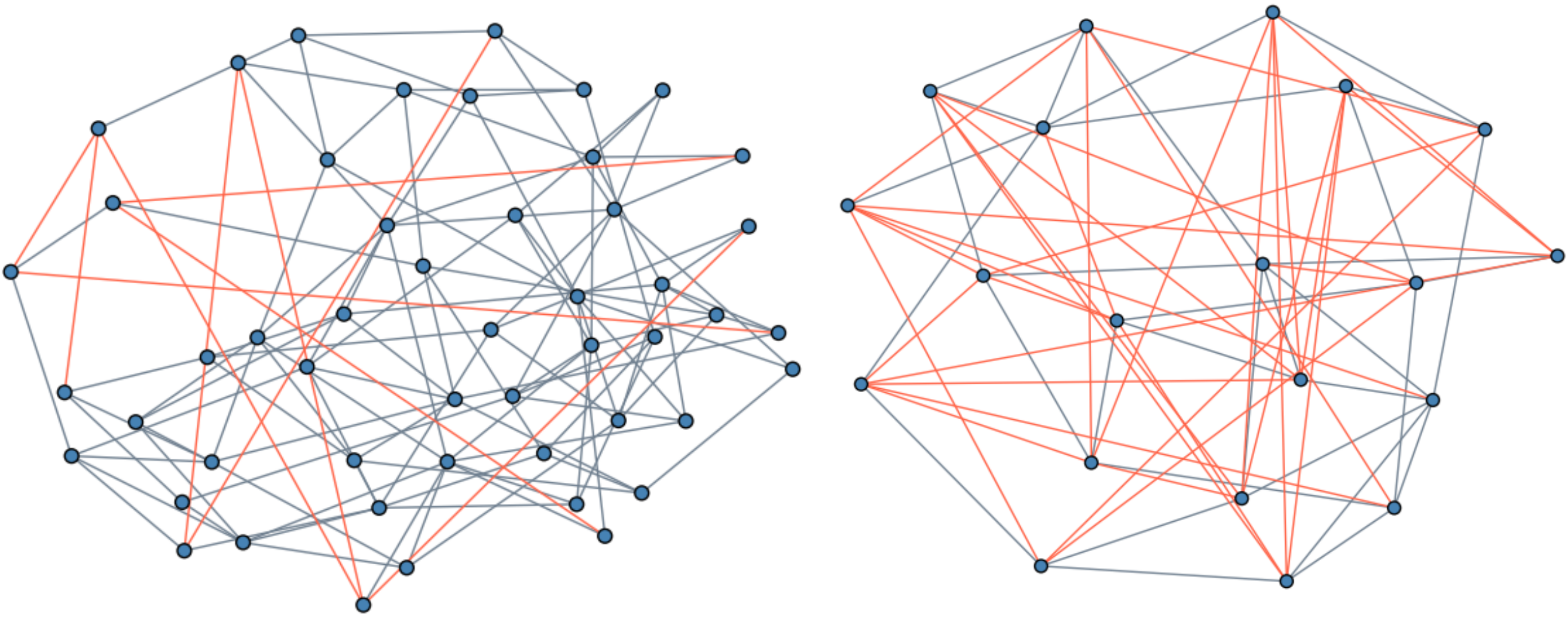}}
\caption{Adding 10 edges to an Erd\H{o}s-R\'{e}nyi random graph on 50 nodes (left); adding 35 edges to an Erd\H{o}s-R\'{e}nyi random random graph on 20 nodes (right). In small added edge sets, the added edges tend to connect distant low-degree vertices, and in large added edge sets, the result converges to a regular graph with small-world-like long distance connections.}  \label{fig:erdos}
\end{center}
\end{figure}

\begin{figure}
\begin{center} 
\resizebox{0.93\linewidth}{!}{\includegraphics{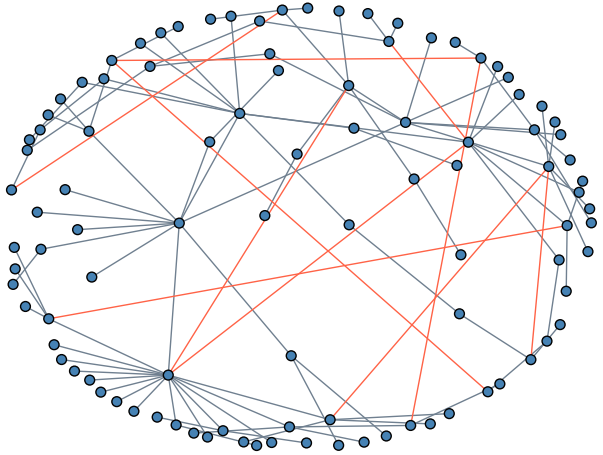}}
\caption{A set of 10 edges added to a Barabasi-Albert scale free network with power law degree distribution. Some of the added edges tend to connect highly connected nodes, while others make low-degree long-distance connections. }  \label{fig:ba}
\end{center}
\end{figure}


\section{Summary and conclusions}
In summary, we considered a network topology design problem in which the objective is to select a set of edges to add to a given graph to optimize the coherence of the resulting network. We showed that network coherence is a submodular function of the network topology, which means that a simple greedy algorithm can be used to select a near optimal edge subset. A modified fast greedy algorithm was developed using Sherman-Morrison pseudoinverse updates and exploiting the submodularity property and provides a computational speed-up of several orders of magnitude, allowing scaling to network sizes far beyond the capabilities of current state of the art semidefinite programming solvers.

Our current work is applying our algorithms to study various network coherence metrics for networks with second-order dynamics and associated power network models and using the algorithms for design of wide-area controllers.

\bibliographystyle{plain}  
\bibliography{refs}  

\appendix{
We prove here via induction that the star graph is the unweighted tree with optimal network coherence.  In other words, in this case each stage of Algorithm \ref{alg1-tree-lapinv} produces a star graph. For the base case, one can directly verify that amongst trees on 4 nodes (i.e., a path and a star), the star is optimal. For the inductive step, assume that a star on $n$ nodes is optimal; we will show that when another node is added the $n+1$ star is also optimal.

Let $L$ be the Laplacian of a star graph of $n$ vertices. Without loss of generality assume 1 is the hub node and $2,\dots,n$ are leaves. The Moore-Penrose pseudoinverse of $L$, $L^\dag$, is given by:
\begin{equation}\label{eq:pinv_entries}
\begin{split}
L^\dag_{1,1}= \dfrac{1-n}{n^2}, \quad L^\dag_{1,i}&=L^\dag_{i,1}= -\dfrac{1}{n^2},\; i=2,\dots,n\\
L^\dag_{i,i}&= \dfrac{n^2-n-1}{n^2},\; i=2,\dots,n\quad  \\L^\dag_{i,j}&=L^\dag_{ji}=	-\dfrac{n+1}{n^2},\; i,j=2,\dots,n.
\end{split}
\end{equation}
Since $L^\dag L=L L^\dag=\dfrac{1}{n} \hat{L}$, where $\hat{L}$ is the Laplacian of a complete graph over vertices $1,\dots,n$, then $I-L L^\dag=\dfrac{1}{n}\mathbf{1}_{n\times n}$, where $\mathbf{1}_{n\times n}\in\mathbb{R}^{n\times n}$ is a matrix of all ones.

Now, consider a graph with $n+1$ vertices where node $n+1$ does not share an edge to any other node and the rest of the vertices form a star graph with node 1 as its root. Denote its Laplacian by $\bar{L}$ where the block formed by the first $n$ rows and columns is matrix $L$ described above and the rest of the entries are zero.
Consider the case where a new edge is to be chosen to connect $n+1$ to any of the vertices so that the trace of the pseudoinverse of the Laplacian of the resulting graph has the smallest value. In other words, it is desired to solve the following optimisation problem:
\begin{equation}\label{eq:opt_one_star}
\begin{split}
\min_{j\in\{1,\dots,n\}} \quad & \mathbf{trace}(\tilde{L}^\dag) \\
\text{s.t.} \quad & \tilde{L}=\bar{L} +m(j)^\top m(j)\\
& m(j)\in\mathbb{Z}^{n+1}, \; m_{n+1}(j)=-1,\; m_j(j)=1,\\
& m_i(j)=0,\; \forall i\in \{1,\dots,n\} \setminus \{j,n+1\}.
\end{split}
\end{equation}
As argued in Section \ref{tree}, 
\begin{equation}
\begin{split}
\mathbf{trace}(\tilde{L}^\dag)&=\mathbf{trace}(\bar{L} +m(j)^\top m(j))^\dag \\
&=  \mathbf{trace}(\bar{L}^\dag)+\dfrac{1+m(j)^\top \bar{L}^\dag m(j)}{\| (I-\bar{L} \bar{L}^\dag)m(j)\|^2}.
\end{split}
\end{equation}
Let $v(j)=(I-\bar{L} \bar{L}^\dag)m(j)$ and $V=I-\bar{L} \bar{L}^\dag$:
$$
v_i(j)=V_{i,j}m_j(j) +V_{i,n+1}m_{n+1}(j)=\begin{cases}
{1}/{n}, & i\neq n+1\\
-1, & i=n+1.
\end{cases}
$$
where $V_{i,j}$ is the $i,j$-th entry of $V$. As a result, $\| (I-\bar{L} \bar{L}^\dag)m(l)\|^2=\| (I-\bar{L} \bar{L}^\dag)m(k)\|^2$ for all $l,k\in\{1,\dots,n\}$. So the value of $m(j)^\top \bar{L}^\dag m(j)$ determines which choice of $j$ results in a smaller trace. Problem \eqref{eq:opt_one_star} can be written as
\begin{equation}\label{eq:opt_one_star_simple}
\begin{split}
\min_{j\in\{1,\dots,n\}} \quad & m(j)^\top \bar{L}^\dag m(j) \\
\text{s.t.} \quad & m(j)\in\mathbb{Z}^{n+1}, \; m_{n+1}(j)=-1,\; m_j(j)=1,\\
& m_i(j)=0,\; \forall i\in \{1,\dots,n\} \setminus \{j,n+1\}.
\end{split}
\end{equation}
Since, the last row and column of $\bar{L}^\dag$ are zeros:
 $$m(j)^\top \bar{L}^\dag m(j) = \bar{L}^\dag_{j,j}.$$
 Remembering \eqref{eq:pinv_entries} it can be seen that for all $n\geq2$:
 $$ \bar{L}^\dag_{1,1}<  \bar{L}^\dag_{i,i},\quad \forall i\in\{2,\dots,n\}.$$ 
 Then the optimisation problem \eqref{eq:opt_one_star} is solved for $j=1$ which means that at each step adding an edge from an isolated node  to the root of an existing star graph is the best strategy to minimize the trace of the pseudoinverse of the Laplacian of the resulting graph.}

\end{document}